\documentclass[final]{siamltex}
\usepackage{amsmath,amssymb}

\newcommand{\cP}{\mathcal{P}}

\newcommand{\R}{\mathbb{R}}

\title{A new approach for the existence problem of
minimal cubature formulas based on the Larman-Rogers-Seidel theorem}

\author{Masatake Hirao\thanks{Graduate School of Information Science,
  Nagoya University, Furo-cho, Chikusa-ku, Nagoya 464-8601, Japan
 ({\tt hirao@math.cm.is.nagoya-u.ac.jp}).}\and
 Hiroshi Nozaki\thanks{Graduate School of Information Sciences,
 Tohoku University, 6-3-09 Aramaki-Aza-Aoba, Aoba-ku, Sendai 980-8579, Japan
 ({\tt nozaki@ims.is.tohoku.ac.jp}).}\and
 Masanori Sawa\thanks{Graduate School of Information Science,
 Nagoya University, Furo-cho, Chikusa-ku, Nagoya 464-8601, Japan
 ({\tt sawa@is.nagoya-u.ac.jp}).}\and
 Vesselin Vatchev\thanks{Department of Mathematics,
 University of Texas, 80 Fort Brown, Brownsville, TX 78520, USA
 ({\tt vesselin.vatchev@utb.edu}).}}

\begin{document}

\maketitle

\begin{abstract}
In this paper we consider the existence problem 
of cubature formulas of degree $4k+1$ for spherically symmetric integrals
for which the equality holds in the M\"oller lower bound.
We prove that for sufficiently large dimensional minimal formulas,
any two distinct points on some concentric sphere
have inner products all of which are rational numbers.
By applying this result
we prove that for any $d \ge 3$
there exist no $d$-dimensional minimal formulas 
of degrees $13$ and $21$ for some special integral.
\end{abstract}



\pagestyle{myheadings}
\thispagestyle{plain}
\markboth{M. HIRAO, H. NOZAKI, M. SAWA, AND V. VATCHEV}
{NEW APPROACH FOR EXISTENCE OF MINIMAL FORMULAS}

\section{Introduction}
\label{sect:1}
A main problem of numerical integration is to approximate the integral
\begin{equation*}
\int_\Omega f(x) {\rm d} \mu.
\end{equation*}
Here $x$ is 
a $d$-dimensional vector and $\mu$ is a positive measure 
on a domain $\Omega$ in $\mathbb{R}^d$.
We search for an approximant by taking a positive linear combination 
of the function values of $f$ 
as specified points $X = \{ x_{1}, \ldots, x_{N} \}$,
that is,
\begin{equation}
\label{eq:CF}
\sum_{i = 1}^N w_i f(x_i).
\end{equation}
We call (\ref{eq:CF}) a {\it cubature formula}.
The values $w_i$ are the {\it weights} and $x_i$ are
the {\it points} of a cubature formula.
From the practical viewpoint of numerical computation,
it is desirable that every coefficient is positive; see~\cite[p.12]{S71}.
To each formula we assign the set of functions for which it is exact.
Most often this set is the space of all polynomials of degree at most $t$;
in this case a cubature formula is said to be {\it of degree $t$}.
We refer the readers to the comprehensive monographs~\cite{DX01,SV97}
for the basic theory of cubature formula.
Hereafter we only consider cubature formulas of odd degree.

It is well known~\cite{R48} that
the smallest possible number of points $|X|$
in a cubature formula of degree $2k+1$ is bounded from below:
\begin{equation}
\label{eq:naturalbound}
|X| \ge \dim \mathcal{P}_k (\Omega).
\end{equation}
Here $\mathcal{P}_k (\Omega)$ is the space of
polynomials of degree at most $k$ restricted to $\Omega$.
The bound (\ref{eq:naturalbound}) can be improved for some broad class of 
integrals called {\it spherically symmetric integrals}~\cite{M79,My75}
(see Section~\ref{sect:2} for the definition of spherical symmetry):
\begin{equation}
\label{eq:Moller}
|X| \ge \left\{\begin{array}{ll}
2 \dim \mathcal{P}_k^* (\Omega) - 1 
& \text{if $k$ is even and $0 \in X$}, \\
2 \dim \mathcal{P}_k^*  (\Omega)
& \text{otherwise}.
\end{array} \right.
\end{equation}
Here $\mathcal{P}_k^* (\Omega)$ is the subspace of $\mathcal{P}_k (\Omega)$
consisting of even polynomials or odd polynomials
if $k$ is even or odd, respectively.
A cubature formula is said to be {\it minimal} 
if equality holds in one of the above lower bounds.

A fundamental problem in combinatorics and numerical analysis
is the existence and nonexistence of
minimal formulas for spherically symmetric integrals.
In the one-dimensional case there always exists a minimal formula
of odd degree, and the points and the weights of the formula
can be uniquely determined in terms of roots of
an orthogonal polynomial~\cite{DX01}.
This one-dimensional formula is known as Gaussian quadrature formula.
The problem becomes much harder in multidimensional cases.
In the two-dimensional case
many minimal formulas have been already 
found~\cite{Baj06,H75,S71,VC92,X98}.
An outstanding result is due to Xu~\cite{X98}:
He found some special spherically (circularly) symmetric integral
for which there exists a minimal formula of {\it general degree}.
In higher dimensional cases, however,
only a few minimal formulas have been found so far,
all of which have small degrees or low dimensions.
It seems to be
the conventional belief (see, e.g.,~\cite{B10,DS89,HS09,NS88,NS04,X03})
that there exists no minimal formula
of degree $4k+1$ for any $d$-dimensional
spherically symmetric integrals for $k \ge 1$ and $d \ge 3$
with some possible exceptional cases.
There are some reasons why
we focus on the degree $4k+1$ case:
The structure of minimal formulas is strongly restricted.
For instance
a famous theorem by M\"oller~\cite{M76} states that
the points of a minimal formula have the antipodality and contain the origin.

There are some recent papers involving small degree cases which
support the conventional belief mentioned above.
Hirao and Sawa~\cite{HS09} observed the relationship
between minimal formulas of degree $5$ and
certain combinatorial objects called tight spherical $5$-designs, and 
thereby
showed nonexistence of minimal formulas of degree $5$ with possible 
exceptions
$d = (2m + 1)^2 - 2$, $m \ge 5$;
see also~\cite{NS04,V04}.
In the same paper they also proved there exists
no minimal formula of degree $9$ for
classical integrals like,
Gaussian integral on $\mathbb{R}^d$,
Jacobi integral on the ball.
Bannai and Bannai~\cite{BB10} improved this result
for all spherically symmetric integrals and
concluded the conventional belief is true for the degree $9$ case.
However as far as the authors know,
very little is known on the existence and nonexistence of minimal formulas
of degree at least $13$ for general values of $d \ge 3$.

The present paper has several aims.
The primary one is to develop
a new theory to number-theoretically characterize
the structure of minimal formulas of degree $4k+1$
for spherically symmetric integrals.
To do this,
we unify a theorem by Mysovskikh~\cite{My81} which is
classical in numerical analysis and
the Larman-Rogers-Seidel theorem~\cite{LRS77} (and
its generalization by Nozaki~\cite{N10}).
The Larman-Rogers-Seidel theorem is widely accepted in
combinatorics, but is not fully recognized
in numerical analysis and related areas.
Conversely,
the Mysovskikh theory is not fully understood by
researchers in combinatorics.
Thus the secondary aim of this paper is to inform
the Larman-Rogers-Seidel theorem to researchers
in numerical analysis, and conversely,
to let combinatorialists know
details of the Mysovskikh theory.
The final aim is to discuss the usefulness of our new theory.

The present paper is organized as follows.
In Section 2 we review some basic facts about minimal formulas and
give preliminary lemmas for further arguments
in the following sections.
In Section 3 by bringing the Mysovskikh theorem and
a generalization of the Larman-Rogers-Seidel theorem together,
we prove that
for any spherically symmetric integral,
the points of a minimal formula of degree $4k+1$
on a particular concentric sphere
have inner products all of which are rational numbers.
In Section 4 we discuss applications of the new theory
to the existence problem of minimal formulas:
We prove that for any $d \ge 3$,
there exist no $d$-dimensional minimal formulas of degree $13$ and $21$
for some particular integral.

\section{Preliminaries}
\label{sect:2}
For $p, q \in \R$ with $0 \leq p < q \leq \infty$,
let $\Omega = \{ x \in \R^{d} \mid p \leq \| x \| < q \}$ 
and $W$ be a
positive weight function on $\Omega$ 
which is a function of $\| x \|$. 
Under the assumption that polynomials up to 
a sufficient large degree are integrable with respect to $\Omega$ for $W$, 
we consider the integral defined by
\begin{equation}
\label{eq:ssi2}
\int_{\Omega} f (x) W(\| x \|) {\rm d} x
= \int_{p}^{q} \left ( \int_{S^{d-1}} f (r x) {\rm d} \rho \right ) r^{d-1} 
W (r) {\rm d} r,
\end{equation}
where $\rho$ is the surface measure on the $d$-dimensional unit sphere 
$S^{d-1}$.
This is called a {\it spherically symmetric integral}.
Several classical integrals like,
Gaussian integral on the entire space,
Jacobi integral on the ball,
belong to the class of spherically symmetric integrals.
Cubature formulas for spherically symmetric integrals find
applications in various fields;
for instance Lyons and Victoir~\cite{LV04} recently considered
cubature formula on Wiener space for application to mathematical finance, 
where cubature for Gaussian integral plays a key role.
Hereafter we only consider spherically symmetric integrals.

Let $l$ be a nonnegative integer.
Let ${\rm Hom}_{l} (\R^{d})$ be the space of all homogeneous polynomials of 
degree $l$ 
and $r_{l}^{d} = \dim ({\rm Hom}_{l} (\R^{d}))$.
We consider an orthonormal basis of ${\rm Hom}_{l} (\R^{d})$, 
say $\{P_{l ,i} \mid 1 \leq i \leq r_l^d \}$, 
and let $\mathbb{P}_{l} = (P_{l, 1},\ldots,P_{l ,r_l^d})$.
The bilinear form
\begin{equation*}
K_{2k} (x,y) 
= \sum_{m = 0}^{k} \mathbb{P}_{2m} (x) \mathbb{P}_{2m}^{\rm T} (y).
\end{equation*}
is called {\it the reproducing kernel of $\cP^{*}_{2k} (\R^{d})$ 
with respect to $\Omega$ for $W$}.
We note that any finite dimensional Hilbert space 
has a unique reproducing kernel; for instance see~\cite{DX01}.  

There is a close relationship between 
minimal formulas of odd degree and reproducing kernels,
as the following famous theorem by Mysovskikh shows~\cite{My81};
we write the statement only for the degree $4k+1$ case.
\begin{theorem}[\cite{My81}]
\label{thm:Mysovskikh}
Assume there exists a minimal formula $X$ of degree $4k+1$ 
for a spherically symmetric integral.
Then, the formula is minimal if and only if the following hold:
\begin{enumerate}
\item[{\rm (i)}] 
$\forall x, y \in X,\; x \ne \pm y,\quad K_{2k}(x, y) =0$.
\item[{\rm (ii)}]
$\forall x \in X \setminus \{0\},\; w(x) = 1/(2 K_{2k}(x, x)) 
\quad \text{and} \quad w(0) = 1/ K_{2k}(0, 0)$.
\end{enumerate}
\end{theorem}
\noindent
We also refer the readers to~\cite{X00} for relationship between minimal 
formulas and reproducing kernels.

Theorem~\ref{thm:Mysovskikh} includes some proven facts that 
the points of a minimal formula of degree $4k+1$
have the antipodality and contain the origin;
see~\cite{M79} for details.
The following result mentions more on the structure of minimal formulas.
\begin{theorem}[\cite{HS09}] 
\label{thm:HS09}
Assume there exists a minimal formula of degree $4k+1$
for a spherically symmetric integral.
Then the following hold:
\begin{enumerate}
\item[{\rm (i)}] The points are distributed over $k$ distinct concentric 
spheres and the origin.
\item[{\rm (ii)}] The weights take a constant on each concentric sphere.
\item[{\rm (iii)}] The points on each concentric sphere are similar to a 
spherical $(2k + 3)$-design.
\end{enumerate}
\end{theorem}
\noindent
See also~\cite{VC92} for the structure of two-dimensional minimal formulas 
of degrees $4k+1$.

To calculate the reproducing kernel of $\cP^{*}_{2k} (\R^{d})$,
we have to compute an orthonormal basis of $\cP^{*}_{2k}(\R^{d})$ 
with respect to $\Omega$ for $W$.
Let ${\rm Harm}(\R^{d})$ be the set of all harmonic polynomials.
Let $l$ be a nonnegative integer.
We define ${\rm Harm}_{l} (\R^{d}) = {\rm Harm}(\R^{d}) \cap {\rm 
Hom}_{l}(\R^{d})$. 
Let $\phi_{l, i},\; i = 1, \ldots, \dim ({\rm Harm}_{l} (\R^{d}))$
be an orthonormal basis of ${\rm Harm}_{l} (\R^{d})$ such that
$$
\frac{1}{|S^{d-1}|} \int_{S^{d-1}} \phi_{l_{1}, i_{1}} (x) \phi_{l_{2}, 
i_{2}} (x) 
{\rm d} \rho = \delta_{l_{1}, l_{2}} \delta_{i_{1}, i_{2}}.
$$
It is well known (see, e.g.,~\cite{BB05, S75}) that
\begin{equation*}
\sum_{i = 1}^{\dim ({\rm Harm}_{l} (\R^{d}))}
\phi_{l,i}(x) \phi_{l,i}(y) = 
\frac{d + 2l - 2}{d - 2} C_{l}^{((d- 2)/2)}(\langle x, y \rangle),  
\quad \forall x, y \in S^{d-1},
\end{equation*}
where $C_{l}^{(\lambda)}$ is the Gegenbauer polynomial 
of parameter $\lambda$, which is explicitly given by 
$$
C_{l}^{(\lambda)} (t) = \frac{(-1)^{l}}{2^{l}}
\frac{\Gamma (\lambda + 1/2) \Gamma (l + 2 \lambda)}
{\Gamma (2 \lambda) \Gamma (l + \lambda + 1/2)} 
\frac{(1 - t^{2})^{1/2 - \lambda}}{l!} \frac{d^{l}}{d t^{l}} (1 - t^{2})^{l 
+ \lambda -1/2}. 
$$
Let $k, m$ be nonnegative integers with $k \geq m$.
We consider the polynomial space $\cP^{*}_{2k -2m} (\R)$ 
with respect to the interval $[p, q)$
for the weight function $V_{0}^{-1} r^{d + 4m - 1} W(r)$,   
where $V_{0} = \int_{p}^{q} r^{d - 1} W(r) {\rm d} r$.
Since $\cP^{*}_{2k - 2m} (\R)$ consists only of even polynomials, 
by using Gram-Schmidt's orthonormalization method, 
we can construct an orthonormal basis $g_{2m, j} (r^{2})$, 
$j = 0, \ldots, k-m$, where $g_{2m, j} (r^{2})$
 is a polynomial of degree $j$ in $r^{2}$, such that 
\begin{equation*}
\delta_{j, j'} = 
\frac{1}{V_{0}} \int_{p}^{q} g_{2m, j} (r^{2}) g_{2m, j'} (r^{2}) 
r^{d + 4m - 1} W(r) {\rm d} r. 
\end{equation*}
By recalling Eq.(\ref{eq:ssi2}), we obtain
\begin{align*}
& \int_{\Omega}
(V_{0} |S^{d-1}|)^{-1/2} \phi_{2 m_{1}, i_{1}} (x) g_{2 m_{1}, j_{1} 
}(\|x\|^{2}) 
\cdot 
(V_{0} |S^{d-1}|)^{-1/2} \phi_{2 m_{2}, i_{2}} (x) g_{2 m_{2}, j_{2} 
}(\|x\|^{2}) 
W(\|x \|) {\rm d} x \\
=& 
\frac{1}{|S^{d-1}|} \int_{S^{d-1}} \phi_{2 m_{1}, i_{1}} (x) \phi_{2m_{2}, 
i_{2}} (x) {\rm d} \rho
\cdot \frac{1}{V_{0}} \int_{p}^{q} g_{2m_{1}, j_{1}} (r^{2}) g_{2m_{2}, 
j_{2}} (r^{2}) 
r^{d + 4m - 1} W (r) {\rm d} r  \\
=& 
\delta_{i_{1}, i_{2}} \delta_{j_{1}, j_{2}} \delta_{m_{1}, m_{2}}.
\end{align*}
Hence
$(V_{0} |S^{d-1}|)^{-1/2} \phi_{2 m, i} (x) g_{2 m, j }(\|x\|^{2})$ 
form an orthonormal basis  
of $\cP^{*}_{2k} (\R^{d})$ with respect to $\Omega$ for $W$;
see~\cite{BB05, HS09}.
\begin{lemma}[\cite{HS09}]
\label{lem:kernel}
Let $d \geq 3$.
The reproducing kernel of $\cP^{*}_{2k} (\R^{d})$ 
with respect to $\Omega$ for $W$ is given as follows: 
\begin{align*}
& K_{2k} (x, y)  \\
&=  \left\{ \begin{array}{ll}
\displaystyle
\frac{1}{V_{0} |S^{d-1}|} 
\sum_{m = 0}^{k} \sum_{j = 0}^{k - m} 
g_{2 m, j }(\|x\|^{2}) g_{2 m, j }(\|y\|^{2}) (\| x \| \| y \|)^{2m} 
\frac{d + 4 m - 2}{d - 2} 
C_{2m}^{((d-2)/2)} \bigg ( \frac{\langle x, y \rangle}{\| x \| \| y \|} 
\bigg) 
& \text{if $x, y \in \R^{d} \setminus \{ 0 \}$,} \\
\displaystyle
\frac{1}{V_{0} |S^{d-1}|} 
\sum_{j = 0}^{k} 
g_{0, j }(\|x\|^{2}) g_{0, j }(\|y\|^{2})
& \text{if $x = 0$}. \\
\end{array} \right.
\end{align*} 
\end{lemma}

The following lemma is useful for
calculating reproducing kernels in 
Section~\ref{sect:4}.
\begin{lemma}
\label{lem:SV}
Let $\mathcal{K}_{1}$ be the reproducing kernel of $\cP_{2k}^{*}(\R)$
with respect to $[1, \infty)$ for a weight function $\mu$ and 
$\mathcal{K}_{2}$ be the reproducing kernel of $\cP_{2k}^{*}(\R)$
with respect to $(0, 1]$ for a weight function $\gamma(y)= y^{- 4k - 2} 
\mu(1/y)$.
Then for any $x, y \in (0,1]$ we have
$$
\mathcal{K}_2(x,y)=(xy)^{2k} \mathcal{K}_1 
\left(\frac{1}{x},\frac{1}{y}\right),
$$
or equivalently, for $x, y \in [1,\infty)$,
$$
\mathcal{K}_1(x,y) = (xy)^{2k} \mathcal{K}_2 
\left(\frac{1}{x},\frac{1}{y}\right).
$$
\end{lemma}

\begin{proof}
Let $R_{j} (x^{2}), j=0,\ldots, k$ be polynomials of degree $j$ in $x^{2}$,   
which form an orthonormal basis of $\cP_{2k}^{*} (\R)$
with respect to $[1, \infty)$ for $\mu$.
Then by changing variable $x = 1/y$, we get
$$
\delta_{j,j'} 
= \int_1^{\infty} R_{j}(x^{2})R_{j'}(x^{2}) \mu (x) {\rm d}x
=\int_0^1 \left (y^{2k} R_j \left (\frac{1}{y^{2}} \right) \right) 
 \left (y^{2k} R_{j'} \left (\frac{1}{y^{2}} \right) \right) 
\frac{ \mu (1/y )}{y^{4k + 2}} {\rm d}y.
$$
Let $P_j (y)= y^{2k} R_j (1/ y^{2})$. 
Then $P_{j} (y) \in \cP_{2k}^{*}(\R)$, and from the assumption
$\gamma(y)= y^{- 2 - 4k} \mu(1/y)$, 
$\int_0^1P_j (y)P_{j'} (y) \gamma(y) {\rm d} y = \delta_{j, j'}$, i.e.,
$P_j$ are an orthonormal basis of $\cP_{2k}^{*} (\R)$
with respect to $(0, 1]$ for $\gamma$
and hence 
$$
\mathcal{K}_2 (x,y) 
= \sum_{j=0}^{k} P_j( x)P_j (y) 
= (xy)^{2k} \sum_{j=0}^{k} R_j \left (\frac{1}{x^{2}} \right)R_j \left( 
\frac{1}{y^{2}} \right).
$$
Since $x,y\in (0,1]$, we have $1/x^{2},1/y^{2} \in [1,\infty)$.
Hence from the fact $R_{j} (1/x^{2})$ are an orthonormal basis of 
$\cP_{2k}^{*} (\R)$
with respect to $[1, \infty)$ for $\mu$, 
it follows that $\sum_{j=0}^{k} R_j (1/x^{2}) R_j(1/y^{2})$  
is a reproducing kernel of $\cP_{2k}^{*} (\R)$
with respect to $[1, \infty)$ for $\mu$.
The result follows by the uniqueness of reproducing kernels of Hilbert space.
\qquad
\end{proof}

\section{Characterizing the structure of minimal formulas}
\label{sect:3}
In this section we develop a theory of minimal formulas of degree $4k+1$ by unifying
the Mysovskikh theorem and a generalization of 
the Larman-Rogers-Seidel theorem~\cite{N10}.

First we introduce the Larman-Rogers-Seidel theorem \cite{LRS77}. 
\begin{theorem}
Let $X$ be a finite set in $\mathbb{R}^d$ with only two Euclidean distances
$a_1, a_2$ ($a_1<a_2$) between any two distinct points. 
If $|X|$ is greater than $2d+3$, 
then $a_1^2/a_2^2=(m-1)/m$ where $m$ 
is a positive integer bounded above by $1/2+\sqrt{d/2}$. 
\end{theorem}

Nozaki~\cite{N10} extended this theorem
for points with more Euclidean distances. 
\begin{theorem} \label{thm:LRStype}
Let $X$ be a finite set in $\mathbb{R}^d$ with only 
$s$ Euclidean distances 
$a_1,\ldots,a_s$ 
between any two distinct points. 
If $|X| \geq 2 \binom{d+s-1}{s-1}+2 \binom{d+s-2}{s-2}$, 
then for each $i = 1,\ldots s$,
\[
\prod_{j=1,2,\ldots,s, j \ne i} \frac{a_j^2}{a_j^2-a_i^2}
\]
is an integer, whose absolute value is bounded by some function of $d$ and 
$s$. 
\end{theorem}

Now, let $r_{1}, \ldots, r_{p}$ be positive real numbers. 
Let $Y$ be a set in $\R^{d}$ supported by
$S_{r_1}^{d-1}, \cdots, S_{r_{p}}^{d-1}$, 
where $S_{r_i}^{d-1}$ is the $(d-1)$-dimensional sphere
of radius $r_i$.
We denote the {\it $i$th layer} of $Y$ by $Y_{i}$, namely,  
$Y_{i} = Y \cap S_{r_i}^{d-1}$. 
Let $A(Y_i) = \{\langle x, y\rangle / r_i^2 \mid x, y \in Y_i, \; x \ne y 
\}$.

We can prove various theorems similar to Theorem \ref{thm:LRStype} 
for antipodal finite sets which are supported by a sphere.
\begin{theorem}[\cite{N10}] \label{thm:anti}
Let $s$ be an integer with $s \geq 4$. 
Let $Y$ be an antipodal finite set in $S_{1}^{d-1}$ with
$s$ inner products between any two distinct points.
If $|Y| \geq 4 \binom{d+s-3}{s-2}+2$, then every inner product is rational. 
\end{theorem}  

Theorem \ref{thm:anti} leads to
the following proposition which is
closely related to minimal formulas of degree $4k+1$.

\begin{proposition}
\label{prop:structure0}
Let $p,d$ be integers such that $p \geq 2, d \geq 4p^{2} - 2p + 1$.
Let $Y$ be an antipodal set supported by $p$ concentric spheres with 
$|Y| = 2 \sum_{i=0}^{p-1} \binom{d + 2p - 1 - 2i}{2p - 2i}$
and for each $i$, $|A (Y_{i})| \leq 2p + 1$. 
Then there exists $Y_{l}$ such that for any $\alpha \in A(Y_l)$, 
$\alpha$ is a rational number. 
\end{proposition}

\begin{proof}
By using the pigeonhole principle, we can choose 
$Y_{l}$ such that
\begin{equation*} 
|Y_{l}| \ge  \frac{2}{p}
\sum_{i=0}^{p-1} \binom{d + 2p - 1 - 2i}{2p - 2i}.
\end{equation*}
By the assumptions of $p$ and $d$,
we obtain 
\begin{align}
\frac{2}{p} \sum_{i=0}^{p-1} \binom{d + 2p - 1 - 2i}{2p - 2i}
& \ge \frac{2}{p}  
\Big( \binom{d + 2p - 1}{2p} + \binom{d + 2p - 3}{2p - 2}  \Big) 
\nonumber \\
& \ge 4 \binom{d+2p-2}{2p-1}+2.
\label{eq:rational0}
\end{align}
If $|A(Y_l)|\leq 2p$,  then $|Y_l| \leq 2 \binom{d+2p-2}{2p-1}$ \cite{DGS77}.
This contradicts (\ref{eq:rational0}), and hence $|A(Y_l)|=2p+1$.
Thus the result follows from Theorem \ref{thm:anti}. 
\qquad\end{proof}

The following theorem presents a number-theoretic characterization of
the structure of minimal formulas of degree $4k + 1$.
\begin{theorem}
\label{thm:structure}
Let $k,d$ be integers such that
$k \ge 2, d \ge 4k^{2} - 2k + 1$.
Assume
there exists a $d$-dimensional minimal formula $X$ of degree $4k + 1$ 
for a spherically symmetric integral.
Then there exists a layer $X_{l}$ such that
every $\alpha \in A(X_l)$ is a rational number.
\end{theorem}
 
\begin{proof}
By Theorem~\ref{thm:HS09}
the points $X$ are distributed over $k$ distinct concentric spheres and the origin.
By (\ref{eq:Moller}) and \cite[Lemma 1.7]{B06},
we have
$$
|X \setminus \{ 0 \}| = 2 \sum_{i=0}^{k-1} \binom{d + 2k - 1 - 2i}{2k - 2i}.
$$
By Lemma~2.3 in~\cite{HS09}, 
the reproducing kernel of $\cP_{2k}^{*}(\R^{d})$ is a polynomial of degree $k$
in variables $\| x \|^2, \| y \|^2, \langle x, y \rangle^2$.
If $x, y \in X,\;x \ne \pm y$, lie on the same layer,
then by Theorem~\ref{thm:Mysovskikh}, 
$K_{2k}(x,y)$ is regarded as a polynomial of degree at most $k$
in one variable $\langle x, y \rangle^2$.
This implies $|A(X_{i})| \leq 2k + 1$ for each layer $X_i$.
Hence the result follows by Proposition~\ref{prop:structure0}. 
\qquad\end{proof}

{\em Remark}.
Bannai and Damerell \cite{BD80} proved that
there exists no tight spherical $t$-design on
$S^{d-1}$ for $t \geq 8$ and $d \geq 4$
except for some examples.
A key idea of their proof is to note that
the inner products of points of a tight spherical design
form the zeros of a certain orthogonal polynomial, 
and moreover to show the rationality of
these zeros by the theory of association scheme.
In this sense Theorem~\ref{thm:structure} is similar to
Bannai and Damerell's approach for Euclidean design.

The usefulness of Theorem~\ref{thm:structure} will be clear in Section~\ref{sect:4}.

\section{Applying the basic theory}
\label{sect:4}
Here we use the same notations as in the previous sections.
We consider the following integral:
\begin{align}
\label{eq:ssi}
\int_\Omega f(x) W(\|x\|) {\rm d} x
&= \int_{ \{ x \in \mathbb{R}^d \mid 1 \le \| x \| < \infty \}}
f(x)  \frac{\sqrt{\| x \|^2-1}}{\| x \|^{4k+d+2}} {\rm d} x \nonumber \\
&= \int_1^\infty \bigg(\int_{S^{d-1}} f(r x) {\rm d} \rho \bigg)
\sqrt{r^2 - 1} \ r^{-(4k+3)} {\rm d} r.
\end{align}
This can be seen as a multidimensional generalization of 
the two-dimensional integral Xu~\cite{X98} focused on,
though, of course, many other generalizations can also be considered.
The purpose of this section is to prove the following theorem.

\begin{theorem}
\label{thm:main}
There exists no $d$-dimensional minimal formula of degree $13$ and $21$
for the integral~(\ref{eq:ssi}) for any $d \ge 2$.
\end{theorem}

To prove Theorem~\ref{thm:main} we use some lemmas.

\begin{lemma} 
\label{prop:repro}
Let $d \geq 3$.
The reproducing kernel of $\cP^{*}_{2k} (\R^{d})$ with respect to 
(\ref{eq:ssi})
is given as follows:
$$
K_{2k} (x, y) 
= \frac{4 (\| x\| \| y \|)^{2k}}{\pi (d - 2) |S^{d-1}|} 
\sum_{m = 0}^{k} \sum_{j = 0}^{k - m}
(d + 4m - 2)
C_{2j}^{(1)} \bigg ( \frac{1}{\| x \|} \bigg )  C_{2j}^{(1)} \bigg ( 
\frac{1}{\| y \|} \bigg )
C_{2 m}^{((d -2)/2)} \bigg ( \frac{\langle x, y \rangle}{\| x \| \| y \|} 
\bigg ), 
$$
for any $x, y \in \R^{d} \setminus \{ 0 \}$.
\end{lemma}

\begin{proof}
We use the same notation $g_{2m,j}$ as in Section~\ref{sect:2}.
We consider the polynomial space $\cP_{2k - 2m}^{*}(\R)$ 
with respect to the interval $(0, 1]$ for the weight function $\sqrt{1 - 
r^{2}}$.
By the orthogonality of the Gegenbauer polynomial, 
$2 \sqrt{V_{0}/\pi} C^{(1)}_{2 j} (r)$  
is an orthonormal basis 
of $\cP_{2k - 2m}^{*}(\R)$ with respect to $(0, 1]$ for $\sqrt{1 - r^{2}}$.
Applying Lemma~\ref{lem:SV}, we have
$\sum_{j = 0}^{k - m} g_{2m, j} (s^{2}) g_{2m, j} (t^{2}) = 
4 \pi^{-1} V_{0} (s t)^{2 (k - m)} 
\sum_{j = 0}^{k - m} C^{(1)}_{2j} (1/s) C^{(1)}_{2j} (1/t)$.
Hence the desired result follows by Lemma~\ref{lem:kernel}.
\qquad \end{proof}

The following lemma determines the layers
of minimal formulas for the integral (\ref{eq:ssi}).
\begin{lemma}
\label{prop:radius}
Assume
there exists a minimal formula of degree $4k+1$ for the integral 
(\ref{eq:ssi}).
Then the points are distributed over the origin and 
$k$ concentric spheres of radii
$( \cos (i\pi/(2k+1)) )^{-1}$
and the sums of weights on each concentric spheres are 
$\pi (2k + 1)^{-1}|S^{d-1}| ( \sin (i \pi/(2k + 1)) )^{2}
(\cos (i \pi/(2k + 1)) )^{4k}$,
where $i = 1, \cdots, k$.
\end{lemma}

{\em Proof}.
We note that for each $\ell = 0, \cdots, 2k$,
\begin{align}
\int_\Omega \|x\|^{2\ell} W (\| x \|) {\rm d} x
=& \int_{S^{d-1}} {\rm d} \rho \ \int_1^\infty r^{2\ell} \ \sqrt{r^2 - 1} \ 
r^{-(4k+3)} 
{\rm d} r \nonumber \\
=& \frac{|S^{d-1}|}{8} \ \int_{-1}^1 \bigg(\frac{t+1}{2}\bigg)^{2k-\ell-1} 
\ \sqrt{1 - t^2} {\rm d} t,
\label{eq:transform1}
\end{align}
where the last equality follows by standard calculations 
including changing variable $r \rightarrow \sqrt{2/(t+1)}$.
By Theorem~\ref{thm:HS09}
the points $X$ are distributed over the origin and
$k$ concentric spheres $S_{r_1}^{d-1},\cdots,S_{r_k}^{d-1}$.
Let $\Lambda_i = \sum_{x \in X \cap S_{r_i}^{d-1}} w(x)$ and $R_i = r_i^2$.
Then the leftmost term of Eq.(\ref{eq:transform1}) can be represented as
$\sum_{i = 1}^{k + 1} \Lambda_{i} R_{i}^\ell$.
Thus 
$$
\int_{-1}^{1} \bigg ( \frac{t + 1}{2} \bigg )^{2k - 1 - l} 
\sqrt{\frac{1 - t}{1 + t}} {\rm d} t = 
\sum_{i = 1}^{k} \bigg (\frac{4}{|S^{d-1}|} \Lambda_{i} R_{i}^{2k} \bigg) 
\big (R_{i}^{-1} \big )^{2k -1 - l}, 
\quad l = 0, \ldots, 2k - 1.
$$
By letting $m = 2k - 1 - l$,
$\tilde R_{i} = (2 - R_{i}) R_{i}^{-1},
\tilde \Lambda_{i} = 4 |S^{d-1}|^{-1} \Lambda_{i} R_{i}^{2k}$, 
this formula is, equivalently,
$$
\int_{-1}^{1} \bigg ( \frac{t + 1}{2} \bigg )^{m} 
\sqrt{\frac{1 - t}{1 + t}} {\rm d} t = 
\sum_{i = 1}^{k} \tilde \Lambda_{i} \bigg (\frac{\tilde R_{i} + 1}{2} \bigg)^{m}, 
\quad m = 0, \ldots, 2k - 1.
$$
This is the Gaussian quadrature formula of degree $2k-1$ 
for the integral 
$\pi^{-1} \int_{-1}^{1} f(t) \sqrt{\frac{1- t}{1 + t}} {\rm d} t$.
It is well known (see, e.g., \cite{K66}) that
the points and weights in the Gaussian quadrature formula of degree $2k-1$
are given explicitly for the univariate integral as follows:
$$
\tilde R_{i} = \cos \frac{2 i \pi}{2k + 1} 
\quad \text{and} \quad
\tilde \Lambda_{i} = \frac{4 \pi}{2k + 1} \left ( \sin \frac{i \pi}{2k + 1} 
\right)^{2},
$$
where $i = 1, \ldots, k$.
By using the half angle formula of the cosine function, 
we obtain 
$R_{i} = r_{i}^{2} = (\cos (i \pi/(2k + 1)))^{-2}$.
Moreover, the sums of weights on each concentric spheres are 
$$
\Lambda_{i} = \frac{|S^{d-1}|}{4} R_{i}^{-2k} \tilde{\Lambda}_{i} = 
\frac{\pi |S^{d-1}|}{2k + 1}  \left ( \sin \frac{i \pi}{2k + 1} \right)^{2}
\left ( \cos \frac{i \pi}{2k + 1} \right)^{4k}. \qquad\endproof
$$ 

{\em Remark}.
It is basically possible to calculate
the radii of layers of minimal formulas
of degree $4k+1$,
by taking $y = 0$ in Theorem~\ref{thm:Mysovskikh}~(i)
and solving the equation of the norm $\| x \|$.
However for ordinary integrals, 
it seems to be rare to get
^^ ^^ compact forms" of the radii of layers
even for small values of $k$.
Here ^^ ^^ compact forms" mean that
the radii take very complicated forms
involving imaginary numbers and power roots, which
are not desirable when we further discuss
the nonexistence of minimal formulas:
For example,
it is useful to know the inner products of points in
minimal formulas when discussing the nonexistence
of minimal formulas. To do so,
finding compact forms of radii is useful as well;
see~\cite{BB10,HS09,VC92} for details.
From the proof of Theorem~\ref{thm:main} below,
the readers will also see
how useful such compact forms are.

{\it Proof of Theorem~\ref{thm:main}}. 
Assume there exists a $d$-dimensional minimal formula $X$ of
degree $4k + 1$ 
for the integral (\ref{eq:ssi}).
Let $x, y \in X_l$ with $x \ne \pm y$
and $A_{\sqrt{R}_l} = \langle x, y \rangle^{2} / R_l^{2}$.
By Theorem~\ref{thm:Mysovskikh}, 
Lemma~\ref{prop:repro} and the half angle formula of the cosine 
function, 
for each $d \geq 3$, we have
\begin{align}
0 =& \sum_{\substack{i + j = k \\ i, j \geq 0}} 
\frac{d + 4j - 2}{4^{k}}
C_{2i}^{(1)} \bigg ( \sqrt{\frac{1 + \cos (2l\pi/(2k + 1))}{2}} \bigg )^{2}
C_{2 j}^{((d -2)/2)} \bigg ( A_{\sqrt{R_{l}}}^{1/2} \bigg ).
\label{eq:proof2}
\end{align}
We regard Eq.(\ref{eq:proof2}) as an equation of $\cos (2 l \pi/(2k + 1))$.

{\it (i) The Case of $k = 3$}.
We assume $d \geq 31$.
According to Theorem~\ref{thm:structure}, 
we can choose some layer of the formula, say $X_{l}$, over which
the points have rational distances.
Let $f(x) = x^3 + \frac{1}{2} x^2 - \frac{1}{2} x - \frac{1}{8} \in 
\mathbb{Q}[x]$ 
be the minimal polynomial of $\cos (2 l \pi/7)$ \cite{WZ93}.
When dividing the right hand side of Eq.(\ref{eq:proof2}) by $f(\cos (2 l \pi/7))$,
the remainder is
\begin{align*}
& \frac{1}{384} d (d + 2) \{ (d + 4) (d + 6) A_{\sqrt{R_l}}^{2}
 - 6 (d + 4)  A_{\sqrt{R_l}} + 3  \}
 \cos \Big( \frac{2 l \pi}{7}\Big)^{2} \nonumber \\
& + \frac{1}{384} d \{ (d + 2) (d + 4) (d + 6) A_{\sqrt{R_l}}^{2}
 - 6 (d + 2) (d + 3) A_{\sqrt{R_l}} + 3 d \}
 \cos \Big( \frac{2 l \pi}{7}\Big)  \nonumber \\
&+ \frac{1}{46080} d \Big \{ 
(d + 2) (d + 4) (d + 6) (d + 8) (d + 10) A_{\sqrt{R_l}}^{3} 
- 15 (d + 2) (d + 4) (d + 6)^{2} A_{\sqrt{R_l}}^{2} \nonumber \\
& \qquad \qquad \qquad + 45 (d + 2) (d^{2} + 6 d + 24) A_{\sqrt{R_l}} - 15 (d^{2} + 44) \Big \}.
\end{align*}
By the minimality of $f$, we have
\begin{align}
0 =& (d + 4) (d + 6) A_{\sqrt{R_l}}^{2}  - 6 (d + 4)  A_{\sqrt{R_l}} + 3. 
\nonumber \\
0 =& (d + 2) (d + 4) (d + 6) A_{\sqrt{R_l}}^{2} - 6 (d + 2) (d + 3) 
A_{\sqrt{R_l}} + 3 d. \nonumber \\
0 =& (d + 2) (d + 4) (d + 6) (d + 8) (d + 10) A_{\sqrt{R_l}}^{3} 
- 15 (d + 2) (d + 4) (d + 6)^{2} A_{\sqrt{R_l}}^{2} \nonumber \\
& \qquad\qquad + 45 (d + 2) (d^{2} + 6 d + 24) A_{\sqrt{R_l}} - 15 (d^{2} + 44). 
\label{eq:proof3}
\end{align}
By multiplying the first equation in (\ref{eq:proof3}) by $d+2$ and 
then subtracting the new equation from the second equation in Eq.(\ref{eq:proof3}),
we see that $A_{\sqrt{R_l}} = 1/(d + 2)$.
Substituting this value of $A_{\sqrt{R_l}}$ into the first equation in 
Eq.(\ref{eq:proof3}) implies that $d = -1, -6$.
This is a contradiction.
We deal with the remaining cases 
$2 \leq d \leq 30$ 
in Appendix~\ref{sect:A}.

{\it (ii) The Case of $k = 5$}.
We assume 
$d \geq 91$.
According to Theorem~\ref{thm:structure},
we can choose some layer of the formula, say $X_{l}$, over which
the points have rational distances.
Let $f(x) = x^{5} + \frac{1}{2} x^{4} - x^{3} 
- \frac{3}{8} x^{2} + \frac{3}{16} x + \frac{1}{32} \in \mathbb{Q}[x]$ 
be the minimal polynomial of $\cos (2 l \pi/ 11)$ \cite{WZ93}.
When dividing the right hand side of Eq.(\ref{eq:proof2}) by $f(\cos (2 l \pi/11))$,
the remainder is
\begin{align*}
c_{0} +
c_{1}  \cos \Big ( \frac{2l \pi}{11} \Big ) +
c_{2}  \cos \Big ( \frac{2l \pi}{11} \Big )^{2} + 
c_{3}  \cos \Big ( \frac{2l \pi}{11} \Big )^{3} +
c_{4}  \cos \Big ( \frac{2l \pi}{11} \Big )^{4},
\end{align*}
where 
\begin{align*}
c_{0} =& \frac{1}{3715891200} d 
\{
 (d + 2) (d + 4) (d + 6) (d + 8) (d + 10) (d + 12) (d + 14) (d + 16) (d + 
18) A_{\sqrt{R_{l}}}^{5}
 \nonumber \\
 & \quad - 45 (d + 2) (d + 4) (d + 6) (d + 8) (d + 10) (d + 12) (d + 
14)^{2} A_{\sqrt{R_{l}}}^{4}
 \nonumber \\
 &\quad + 630 (d + 2) (d + 4) (d + 6) (d + 8) (d + 10) (d^{2} + 22d + 128) 
A_{\sqrt{R_{l}}}^{3}
 \nonumber \\
&\quad  - 3150 (d + 2) (d + 4) (d + 6) (d + 8) (d^{2} + 16 d + 84) 
A_{\sqrt{R_{l}}}^{2}
\nonumber \\
& \quad + 4725 (d + 2) (d^{2} + 6 d + 32) (d^{2} + 14 d + 72) 
A_{\sqrt{R_{l}}}
-945 (d^{4} + 10 d^{3} + 100 d^{2} + 440 d + 4384)
\}, \nonumber \\
c_{1} =& \frac{1}{10321920} d 
\{ (d + 2) (d + 4) (d + 6) (d + 8) (d + 10) (d + 12) (d + 14) 
A_{\sqrt{R_{l}}}^{4}
\nonumber \\
& \quad - 28 (d + 2) (d + 4) (d + 6) (d + 8) (d + 10) (d + 14) 
A_{\sqrt{R_{l}}}^{3}
+ 210 (d + 2) (d + 4) (d + 6) (d + 8) (d + 14) A_{\sqrt{R_{l}}}^{2}
\nonumber \\
& \quad - 420 (d + 2) (d^{3} + 24 d^{2} + 164 d + 312) A_{\sqrt{R_{l}}}
+ 105 (d^{3} + 20 d^{2} + 92 d + 16)
\}, 
\nonumber \\
c_{2} =& \frac{1}{10321920} d (d + 2)
\{ (d + 4) (d + 6) (d + 8) (d + 10) (d + 12) (d + 14) A_{\sqrt{R_{l}}}^{4}
\nonumber \\
& \quad - 28 (d + 4) (d + 6) (d + 8) (d + 10) (d + 14) A_{\sqrt{R_{l}}}^{3}
+ 210 (d + 4) (d + 6) (d^{2} + 22 d + 128) A_{\sqrt{R_{l}}}^{2}
\nonumber \\
& \quad - 420 (d + 4) (d^{2} + 20 d + 132) A_{\sqrt{R_{l}}} 
+ 105 (d^{2} + 18 d + 152)\}, \\
c_{3} =& \frac{1}{46080} d (d + 2) 
\{ (d + 4) (d + 6) (d + 8) (d + 10) A_{\sqrt{R_{l}}}^{3} 
-15 (d + 4) (d + 6) (d + 7) A_{\sqrt{R_{l}}}^{2}
\nonumber \\
&\quad + 45 (d + 4)^{2} A_{\sqrt{R_{l}}} - 15 (d + 1)\}, 
\nonumber \\
c_{4} =&  \frac{1}{46080} d (d + 2) (d + 4)
\{ (d + 6) (d + 8) (d + 10)A_{\sqrt{R_{l}}}^{3} 
-15 (d + 6) (d + 8) A_{\sqrt{R_{l}}}^{2} 
+ 45 (d + 6) A_{\sqrt{R_{l}}} -15\}.
\nonumber \\
\end{align*}
By the minimality of $f$, we have
$
c_{0} = 0, c_{1} = 0, c_{2} = 0, 
c_{3} = 0, c_{4} =0,
$
and therefore
\begin{align*}
f_{3, 4} := c_{3} - c_{4} =& \frac{1}{3072} 
\{ (d^{3} + 12 d^{2} + 44 d + 48) A_{\sqrt{R_{l}}}^{2}
- 6 (d^{2} + 6 d + 8) A_{\sqrt{R_{l}}} + 3 (d + 2) \}, \\
f_{2, 1} := c_{2} - c_{1} =& \frac{1}{3072} 
\{ (d^{3} + 12 d^{2} + 44 d + 48) A_{\sqrt{R_{l}}}^{2}
- 3 (2d^{2} + 13 d + 18) A_{\sqrt{R_{l}}} + 3 (d + 3) \}.
\end{align*}
Since
$$
0 = f_{3, 4} - f_{2,1} = \frac{d \{ (d + 2) A_{\sqrt{R_{l}}} - 1 \}}{1024},
$$
we obtain $A_{\sqrt{R_{l}}} = 1 /(d+2)$.
By substituting $A_{\sqrt{R_{l}}}$ into $f_{3,4}$, we obtain
\begin{align*}
f_{3, 4}  = - \frac{d (d + 1) (d + 6)}{1536 (d + 2)} < 0.
\end{align*}
This is a contradiction.
We deal with the remaining cases $2 \leq d \leq 90$ 
in Appendix~\ref{sect:A}.

{\em Remark}.
When $2k + 1$ is a prime number,
the minimal polynomial of $\cos (2\pi l/(2k+1))$ 
is uniquely determined, 
not depending on $l$ \cite{WZ93}.
Using this fact, we proved
there exist no minimal formulas of
degrees $13$ and $21$ for the integral (\ref{eq:ssi}).
Similar but more extensive considerations will
enable us to discuss the nonexistence of
minimal formulas for other degrees, though
we do not pursue it here.

\section{Conclusion}
\label{sect:5}
Unifying the Mysovskikh theorem and
a generalization of the Larman-Rogers-Seidel theorem,
we obtained a new necessary condition for
the existence of minimal formulas
of degree $4k+1$ for spherically symmetric integrals.
And thereby we solved the existence problem of
minimal formulas of degrees $13$ and $21$
for the integral (\ref{eq:ssi}).
In this paper we only focused on the particular integral.
We hope that
our approach could also be applied
for other spherically symmetric integrals.

\appendix
\section{Small dimensional cases of the proof of Theorem \ref{thm:main}}
\label{sect:A}

In this section we finish off the remaining cases of the proof of Theorem~\ref{thm:main}.
The following lemma plays a key role.
\begin{lemma}
\label{lem:small-dim}
Let $d \geq 3$.
Assume there exists a $d$-dimensional minimal formula 
of degree $4k + 1$ for the integral (\ref{eq:ssi}).
Then, 
\begin{equation}
\frac{8}{2k + 1}
\sum_{m = 0}^{k} \sum_{j = 0}^{k - m}
\bigg ( \sin \frac{(2j + 1) \pi}{2k + 1} \bigg )^{2} 
\bigg \{ \binom{d + 2m - 1}{2m} - \binom{d + 2m - 3}{2m - 2} \bigg\} 
\label{eq:integer}
\end{equation}
is an integer.
\end{lemma}

\begin{proof}
We use the same notations $r_{1}, S_{r_{1}}^{d-1}, \Lambda_{1}$
as in Section~\ref{sect:4}.
According to Theorem~\ref{thm:HS09} (ii), 
let $w_{1} = w(x)$ for every $x \in X \cap S_{r_{1}}^{d-1}$.
By combining Theorem~\ref{thm:Mysovskikh} (ii) and Lemma~\ref{prop:repro}, 
we have $w_{1} = 1/ (2 K_{2k}(x, x))$.
Thus by Lemma~\ref{prop:radius}, 
\begin{align*}
|X \cap S_{r_{1}}^{d-1}| = \frac{\Lambda_{1}}{w_{1}} 
=& 
\frac{8}{2k + 1} \bigg ( \sin \frac{\pi}{2k + 1} \bigg)^{2}
\sum_{m = 0}^{k} \sum_{j = 0}^{k - m} \frac{d + 4m - 2}{d-2}
\bigg \{ C_{2j}^{(1)} \bigg ( \cos \frac{\pi}{2k + 1} \bigg ) \bigg \}^{2} 
C_{2m}^{((d - 2)/2)} (1) \\
=&
\frac{8}{2k + 1}
\sum_{m = 0}^{k} \sum_{j = 0}^{k - m}
\bigg ( \sin \frac{(2j + 1) \pi}{2k + 1} \bigg )^{2} 
\frac{d + 4m - 2}{d -2} C_{2m}^{((d - 2)/2)} (1) \\
=& 
\frac{8}{2k + 1}
\sum_{m = 0}^{k} \sum_{j = 0}^{k - m}
\bigg ( \sin \frac{(2j + 1) \pi}{2k + 1} \bigg )^{2} 
\bigg \{ \binom{d + 2m - 1}{2m} - \binom{d + 2m - 3}{2m - 2} \bigg\}. 
\end{align*}
where the last equality follows by the fact
$\frac{d + 4m - 2}{d - 2} C_{2m}^{((d - 2)/2)} (1) = \dim ({\rm Harm}_{2m} 
(\R^{d}))$ (see~\cite{BB05}). 
\qquad\end{proof}

Let $n$ be an positive integer.
For $j = 0, \ldots, k$,
we obtain approximate sine functions
\begin{align*}
S_{1} (n; k, d) =& \sum_{l = 0}^{n- 1} \frac{(-1)^{2l}}{(4l + 1)!} 
\bigg ( \frac{2j + 1}{2k + 1} \pi_{-} \bigg )^{4l + 1} -
\sum_{l = 0}^{n - 1} \frac{1}{(4l + 3)!} 
\bigg ( \frac{2j + 1}{2k + 1} \pi_{+} \bigg )^{4l + 3}, \\
S_{2} (n; k, d) =& \sum_{l = 0}^{n- 1} \frac{(-1)^{2l}}{(4l + 1)!} 
\bigg ( \frac{2j + 1}{2k + 1} \pi_{+} \bigg )^{4l + 1} -
\sum_{l = 0}^{n - 1} \frac{1}{(4l + 3)!} 
\bigg ( \frac{2j + 1}{2k + 1} \pi_{-} \bigg )^{4l + 3} + 
\frac{(\pi_{+})^{4n}}{(4n)!}, 
\end{align*}
where $\pi_{\pm}$ are positive real numbers with $\pi_{-} < \pi < \pi_{+}$.
These functions
satisfy 
$S_{1} (n; k, d) < \sin ((2j + 1)\pi/(2k + 1)) < S_{2} (n; k, d)$.  
By using Lemma \ref{lem:small-dim} and the above equations, we have
\begin{align*}
N_{1} := & \frac{8}{2k + 1}
\sum_{m = 0}^{k} \sum_{j = 0}^{k - m}
(S_{1} (n;k,d))^{2}
\bigg \{ \binom{d + 2m - 1}{2m} - \binom{d + 2m - 3}{2m - 2} \bigg\} \\
& \; \; < |X_{1}| <
 \frac{8}{2k + 1}
\sum_{m = 0}^{k} \sum_{j = 0}^{k - m}
(S_{2} (n;k,d))^{2}
\bigg \{ \binom{d + 2m - 1}{2m} - \binom{d + 2m - 3}{2m - 2} \bigg\} =: 
N_{2}.
\end{align*}
By choosing the suitable values $n$ and $\pi_{\pm}$, 
we see that
(\ref{eq:integer}) are not integers when 
(i) $k = 3$, $3 \le d \le 30$, (ii) $k =5$, $3 \le d \le 90$.
 
For example, for $(k, d) = (3, 20)$, 
by substituting $n = 5$, $\pi_{+} = 3.14160$ and $\pi_{-} = 3.14159$, 
we obtain 
$$ 47868.2 < N_{1} < |X_{1}| < N_{2} < 47868.8. $$
When
$(k, d) = (5, 80)$, 
by substituting $n = 11$, $\pi_{+} = 3.1415926535898$ and $\pi_{-} = 
3.1415926535897$, 
we obtain
$$
318122993450.96 
< N_{1} < |X_{1}| < N_{2} < 3.18122993450.99.
$$
It remains to consider the case when $d = 2$.
Since the radii and weights of minimal formulas of degree $4k+1$ 
are determined by Lemma~\ref{prop:radius}.  
By applying these values to the equation in \cite[Theorem 3.1.1 (2)]{BBHS10}, 
we easily see that there exists no two-dimensional minimal formula, which
completes the proof.

\subsection*{Acknowledgments}
This research started during the second and 
third author's visit at the University of Texas at Brownsville, 2010,
under the sponsorship of the Japan Society for the Promotion of Science.
They would like to express their sincerest appreciation to 
Oleg Musin for his hospitality and discussions throughout this work.
The first author is deeply grateful to Hiroyuki Matsumoto
for fruitful discussions. 
The third author would also like to thank Eiichi Bannai, Etsuko Bannai, Yuan Xu
for many valuable comments for the present work.

\end{document}